\documentclass{amsart}

\usepackage{graphicx}
\usepackage{tikz-cd}
\usepackage{amssymb}
\usepackage{amsthm}
\usepackage[normalem]{ulem}
\usetikzlibrary{arrows}

\theoremstyle{plain}
\newtheorem{theorem}{Theorem}[section]
\newtheorem{lemma}[theorem]{Lemma}
\newtheorem{proposition}[theorem]{Propostion}
\newtheorem{corollary}[theorem]{Corollary}
\newtheorem*{corollary*}{Corollary}

\theoremstyle{definition}
\newtheorem{definition}[theorem]{Definition}
\newtheorem{example}[theorem]{Example}

\theoremstyle{remark}

\numberwithin{equation}{section}

\begin{document}
\def\Prod{\displaystyle \prod}
\def\Lim{\displaystyle \lim}
\def\Int{\displaystyle \int}
\def\Sum{\displaystyle \sum}
\def\Frac{\displaystyle \frac}
\def\Sup{\displaystyle \sup}
\def\Inf{\displaystyle \inf}
\def\Min{\displaystyle \otimes_{\mathrm{min}}}
\def\Max{\displaystyle \otimes_{\mathrm{max}}}
\def\R{\mathbb{R}}
\def\C{\mathbb{C}}
\def\D{\mathbb{D}}
\def\T{\mathbb{T}}
\def\N{\mathbb{N}}
\def\Q{\mathbb{Q}}
\def\F{\mathbb{F}}
\def\Z{\mathbb{Z}}
\def\H{\mathcal{H}}
\def\A{\mathcal{A}}
\def\I{\mathcal{I}}
\def\x{\tilde{x}}
\def\y{\tilde{y}}
\def\B{\mathbb{B}}
\def\L{\mathcal{L}}
\def\cQ{\mathcal{Q}}
\def\P{\mathcal{P}}
\def\K{\mathcal{K}}
\def\Cstar{\mathrm{C}^*}
\def\id{\text{id}}
\def\Mat{\text{Mat}}
\def\s{\K\otimes}

\title{Universal $\Cstar$-algebras with the Local Lifting Property}

\author{Kristin E. Courtney}
\address{Mathematical Institute, WWU M\"{u}nster, Einsteinstr. 62, 48149 M\"{u}nster, Germany}
\email{kcourtne@uni-muenster.de}

\begin{abstract}
The Local Lifting Property (LLP) is a localized version of projectivity for completely positive maps between $\Cstar$-algebras.
   Outside of the nuclear case, very few $\Cstar$-algebras are known to have the LLP. 
   In this article, we show that the LLP holds for the algebraic contraction $\Cstar$-algebras introduced by Hadwin and further studied by Loring and Shulman.
   We also show that the universal Pythagorean $\Cstar$-algebras introduced by Brothier and Jones have the Lifting Property. 
\end{abstract}


\maketitle

\section{Introduction}
In 1993, Kirchberg introduced the Local Lifting Property (LLP) as a crucial component of his seminal work (\cite{Kir93}) on Connes' Embedding Problem. By his definition, a unital $\Cstar$-algebra $A$ has the Lifting Property (LP) if any completely positive contractive (cpc) map from $A$ into a $\Cstar$-algebra quotient admits a cpc lift. The LLP is a weakening of the LP where the lift is only locally defined. (See Definitions \ref{def liftable} and \ref{defLLP}.)

The Choi-Effros Lifting Theorem (\cite{CE76}) implies that all nuclear $\Cstar$-algebras have the LLP. Outside of nuclear examples, few $\Cstar$-algebras are known to have or to lack the LLP, and such examples are often closely tied to interesting and difficult problems. By proving that $\mathcal{B}(\H)$ fails to have the LLP when $\H$ is infinite dimensional (\cite{JP}), Junge and Pisier resolved a number of conjectures that had been proved to be equivalent by Kirchberg in \cite[Proposition 8.1]{Kir93}. Also, Thom's non-RF hyperlinear property (T) group (\cite{Tho10}) was a highly sought concrete example of a group whose full $\Cstar$-algebra failed the LLP. (Ozawa had already established the existence of groups lacking the LP in \cite{Oza04b}.) As for positive examples, Kirchberg showed in \cite{Kir94} that full group $\Cstar$-algebras for free groups have the LLP, which facilitated their role in his equivalent formulation of Connes' Embedding Problem (\cite[Proposition 8.1]{Kir93}). The LLP is known to be preserved under full free products (\cite{Pis96}), tensoring with nuclear $\Cstar$-algebras (\cite{Kir93}), and passing to relatively weakly injective $\Cstar$-subalgebras (\cite{Kir93}).
However, apart from ones obtained in these constructions, not many concrete examples are known. In fact, it was only very recently that Pisier provided the first example of a non-nuclear $\Cstar$-algebra that has both the LLP and Lance's Weak Expectation Property (\cite{Pis19}), solving a long-standing question implicit in much of Kirchberg's work from the 1990's. 

In this article, we prove that the LLP holds for two more families of $\Cstar$-algebras. 
The first is the family of $\Cstar$-algebras admitting a conditionally projective map (Definition \ref{ELP}) from $\C^N$ for some $N\geq 1$ (Theorem \ref{condproj}). It follows from a result of Loring and Shulman (\cite{LS14}) that, for any $C>0$ and non-constant polynomial $p\in \C[z]$ in one variable, the universal $\Cstar$-algebra for the relations $\|x\|\leq C$ and $p(x)=0$ is included in this family (Corollary \ref{polyLP}). The second is the family of Pythagorean $\Cstar$-algebras $\P_n$, which are universal $\Cstar$-algebras for the relation $\sum_{j=1}^n x_j^*x_j=1$ (Proposition \ref{Pn has LP}); here we actually prove that the LP holds. For the first family, the question of when $^*$-homomorphisms into $\Cstar$-algebra quotients lift to $^*$-homomorphisms has been 
extensively studied in \cite{OP89, Had95, Lor93b, Shu08, Shu12}, among other articles. The Pythagorean $\Cstar$-algebras for $n>1$ were introduced by Brothier and Jones in \cite{BJ} to study so-called Pythagorean representations of Thompsons' groups $F_n$. 

Universal $\Cstar$-algebras can be identified with their defining $^*$-algebraic relations in the sense that, for any operator(s) satisfying the defining relations, there is a natural surjective $^*$-homomorphism from the universal $\Cstar$-algebra onto the $\Cstar$-algebra generated by the operator(s). Some universal $\Cstar$-algebras have nice concrete descriptions; a classical example is the universal $\Cstar$-algebra generated by an isometry, which Coburn showed can be identified with the Toeplitz algebra. 
However, many important and interesting examples of $\Cstar$-algebras, including most of those covered by Corollary \ref{polyLP} and Proposition \ref{Pn has LP}, have no concrete ``Coburn" descriptions. 
Given that cpc maps often do not preserve these defining $^*$-algebraic relations, it is surprising that we can say much about cpc maps coming from these $\Cstar$-algebras. Then again, a cpc map can be dilated to a $^*$-homomorphism via Stinespring's Dilation Theorem, and operators can be dilated to satisfy certain $^*$-algebraic relations, for example Halmos' $2\times 2$ dilation of a contractive operator to a unitary. These two types of dilations underly the proofs of Theorem \ref{condproj} and Proposition \ref{Pn has LP}, respectively. 

In Section \ref{preliminaries}, we establish basic notation and describe the relationships between various lifting properties for $\Cstar$-algebras. In Section \ref{Cond Proj and LLP}, we flesh out a characterization of the LLP (Proposition \ref{LLP}) that comes from Kirchberg's proof of \cite[Lemma 3.3]{Kir94} combined with a characterization of the LLP due to Ozawa (\cite[Proposition 3.13]{Oza04}). Using this, we prove in Theorem \ref{condproj} that any $\Cstar$-algebra with a conditionally projective map $\C^N\to A$ has the LLP. In particular, the LLP holds for the aforementioned universal $\Cstar$-algebras for the relations $\|x\|\leq C$ and $p(x)=0$ (Corollary \ref{polyLP}). Section \ref{Pn} is devoted to proving the LP for Pythagorean $\Cstar$-algebras. As a consequence, we show that the semigroup Ext$(\P_n)$ is a group for each $n>1$ (Corollary \ref{Ext}). Finally, in Section \ref{CEP} we discuss connections between certain universal $\Cstar$-algebras and Lance's Weak Expectation Property (Definition \ref{WEP}). 
\\

\noindent {\bf Acknowledgments.} The author would like to thank James Gabe and David Sherman for many enlightening and fruitful discussions on math and exposition. She would also like to thank Wilhelm Winter for comments on an earlier draft. This research was partially supported by the Deutsche Forschungsgemeinschaft (DFG, German Research Foundation) under Germany's Excellence Strategy -- EXC 2044 -- 390685587, Mathematics M\"{u}nster -- Dynamics -- Geometry -- Structure, and also partially supported by ERC Advanced Grant 834267 - AMAREC.

\section{Preliminaries}\label{preliminaries}

\subsection{Universal $\Cstar$-algebras}
 Let $\mathcal{G}=\{x_i\}_{i\in \mathcal{I}}$ be a set and $\mathcal{R}$ be a set of relations of the form 
 $$\|p(x_{i_1},...,x_{i_n})\|\leq r$$ where $0<n<\infty$, $r\in [0,\infty]$ and $p$ is a noncommutative $^*$-polynomial.
We define the universal $\Cstar$-algebra $\Cstar \langle \mathcal{G}: \mathcal{R}\rangle$ for the relations $\mathcal{R}$ to be the unique $\Cstar$-algebra such that, for any Hilbert space $\H$ and set of operators $\{T_i\}_{i\in \mathcal{I}}\subset \mathcal{B}(\H)$ satisfying $\mathcal{R}$, the assignments $x_i\mapsto T_i$, $i\in \mathcal{I}$ induce a surjective $^*$-homomorphism $\Cstar \langle \mathcal{G}: \mathcal{R}\rangle \to \Cstar (\{T_i\}_{i\in\mathcal{I}})$. 
As long as there exist operators $\{T_i\}_{i\in \mathcal{I}}$ on some Hilbert space that satisfy $\mathcal{R}$, and $\mathcal{R}$ enforces norm bounds on the generators, we are guaranteed that $\Cstar \langle \mathcal{G}: \mathcal{R}\rangle$ exists (see \cite[Theorem 3.1.1]{Lor97}). 
Universal $\Cstar$-algebras can be defined in more generality, but these definitions suffice for the content of this article. In fact, the universal $\Cstar$-algebras we consider here have finitely many generators and relations. The reader is directed to \cite[Chapter 3]{Lor97} or \cite{Bla85} for a more thourough introduction.
 
 The relations can force $\Cstar \langle \mathcal{G}: \mathcal{R}\rangle$ to be unital. (For example, for $\Cstar \langle x : x^*x=1=xx^*\rangle = C(\mathbb{T})$, we implicitly assume $1\in \mathcal{G}$ and the relations making it a unit are in $\mathcal{R}$). To unitize a non-unital universal $\Cstar$-algebra $\Cstar \langle \mathcal{G}: \mathcal{R}\rangle$, it suffices to add $1$ to $\mathcal{G}$ along with appropriate relations (see Section 4 of \cite{CShe}). For a $\Cstar$-algebra $A$, we denote its unitization by $\tilde{A}$, and we take the convention that $\tilde{A}=A$ when $A$ is unital. Because this notation is cumbersome for $\Cstar \langle \mathcal{G}: \mathcal{R}\rangle$, we will denote the unitization of $\Cstar \langle \mathcal{G}:\mathcal{R}\rangle$ by $\Cstar _u\langle \mathcal{G}:\mathcal{R}\rangle$.

\begin{example}
The following are universal $\Cstar$-algebras that will feature in this article.
\begin{enumerate}
    \item The universal unital $\Cstar$-algebra associated to a single polynomial relation, 
    $$\Cstar _u\langle x: \|x\|\leq C,\ p(x)=0\rangle$$ where $C>0$ and $p\in \C[z]$ is a non-constant polynomial that has at least one root of modulus strictly less than $C$.\footnote{We add the requirement that there exists a $\lambda\in \C$ such that $|\lambda|<C$ and $p(\lambda)=0$ in order to guarantee there is at least one representation. See \cite[Remark 5.13]{CShe}.}
    \item The universal (unital) Pythagorean $\Cstar$-algebras for $n\geq 1$,
    $$\P_n:=\Cstar \langle x_1,...,x_n: \textstyle{\sum} x_i^*x_i=1\rangle.$$
    For $n=1$, this is just the Toeplitz algebra by Coburn's Theorem. For $n>1$, these were introduced and studied in \cite{BJ}. 
    \item The universal unital $\Cstar$-algebra generated by a contraction,
    $$\A:=C_u^*\langle x : \|x\|\leq 1\rangle.$$
    This algebra is basic and ubiquitous in the literature, though not always appearing by name. For an extensive study, see \cite{CShe}. 
    \item The 
    universal $\Cstar$-algebra generated by a partial isometry,
    $$\P:=\Cstar \langle x: xx^*x=x\rangle.$$ 
    The defining relation of this algebra is of similar importance. For a study on the $\Cstar$-algebra itself, see \cite{BN12}. 
    \item The universal unital $\Cstar$-algebra for an $n$-row contraction for $n\geq 1$,
    $$\mathcal{R}_n:=\Cstar _u\langle y_1,...,y_n: \|\textstyle{\sum} y_iy_i^*\|\leq 1\rangle.$$
    Row contractions have been of particular interest in generalizations of the commutant lifting theorem and von Neumann's inequality e.g. \cite{Bun84,  Har17, Pop89}. Some facts about the $\Cstar$-algebras themselves were given in \cite{CShe}. 
\end{enumerate}
\end{example}
For most of these $\Cstar$-algebras, there is no version of Coburn's theorem giving a nice concrete realization. Perhaps the only exception from the above list (aside from the Toeplitz algebra) is $$\Cstar \langle x: \|x\|\leq 1, x^2=0\rangle\simeq C_0((0,1], \mathbb{M}_2(\C)),$$ which Loring showed in \cite{Lor93b}.

\subsection{Lifting ucp maps}\label{Prelim: LP}
For $\Cstar$-algebras $A$ and $B$ with $A$ unital and $E\subset A$ an operator system, we say a linear map $\phi:E\to B$ is {\it positive} if it maps positive elements to positive elements. It is called {\it completely positive} (cp) it remains positive under matrix amplifications $\phi^{(n)}:M_n(E)\to M_n(B)$. It is called {\it completely positive contractive} (cpc) if it (and its matrix amplifications) are all norm non-increasing. This is automatic in the case where $B$ and $\phi$ are unital and $\phi$ is completely positive, in which case we call $\phi$ {\it unital completely positive} (ucp). See \cite{Pau02} for an introduction to these maps and their relevant properties. 

\begin{definition}\label{def liftable}
Given $\Cstar$-algebras $A$ and $B$ with $A$ unital and $J\subset B$ a closed two-sided ideal with quotient map $\pi:B\to B/J$, we say a cpc map $\phi:A\to B$ is {\it liftable} if there exists a cpc map $\psi:A\to B$ such that $\phi=\pi\psi$. We say $\phi$ is {\it locally liftable} if for any finite-dimensional operator system $E\subset A$, there is cpc map $\psi:E\to B$ so that $\phi|_E=\pi\psi$. 
\end{definition}

\begin{definition}\label{defLLP}
A unital $\Cstar$-algebra $A$ has the {\it Lifting Property} (LP) (resp. {\it Local Lifting Property} (LLP)) if for every $\Cstar$-algebra $B$ with closed two-sided ideal $J\subset B$, every cpc map $\phi:A\to B/J$ is liftable (resp. locally liftable). 
A non-unital $\Cstar$-algebra $A$ has the (L)LP if and only if its unitization has the (L)LP. 
\end{definition}

In fact, to know a unital $\Cstar$-algebra $A$ has the LLP, it suffices to know that any ucp map from $A$ into a quotient $\Cstar$-algebra is locally liftable; when $A$ is separable, it has the LP provided that every ucp map into a quotient $\Cstar$-algebra is liftable. See \cite[Proposition 13.1.2]{BO} for a proof. 
In \cite{Oza01}, Ozawa gives a particularly useful characterization of the LLP for unital separable $\Cstar$-algebras. The actual proposition we want comes implicitly from \cite[Proposition 2.9]{Oza01} via \cite[Remark 2.11]{Oza01}. He later states the result explicitly, albeit without proof in \cite[Proposition 3.13]{Oza04}, and so this is the citation we use.

\begin{proposition}[{{\cite[Proposition 3.13]{Oza04}}}]\label{Ozawa}
A separable unital $\Cstar$-algebra $A$ has the LLP if and only if any ucp map from $A$ into the Calkin algebra $\mathcal{B}(\ell^2)/\K(\ell^2)$ admits a ucp lift. 
\end{proposition}

Perhaps the most important examples of $\Cstar$-algebras with the (local) lifting property were established by Kirchberg in \cite{Kir94}. 

\begin{lemma}[{{\cite[Lemma 3.3]{Kir94}}}]\label{ArvesonLP}
Let $\F$ be any non-abelian\footnote{The abelian case also holds but is covered by the Choi-Effros lifting theorem.} free group. Then $\Cstar (\F)$ has the LLP. If $\F$ is countably generated, then $\Cstar (\F)$ has the LP. 
\end{lemma}

One should think of $\Cstar$-algebras with the LP as projective objects in the category of unital $\Cstar$-algebras where the morphisms are ucp maps and the epimorphisms are $^*$-homomorphisms. With this perspective, Kirchberg's theorem for $\Cstar (\F)$ can be viewed as a projective analogue of Arveson's celebrated extension theorem, which established the injectivity of $\mathcal{B}(\H)$ for any infinite dimensional Hilbert space $\H$. 

Dual to the LLP is Lance's Weak Expectation Property, which is also sometimes referred to as ``weak injectivity". For the sake of brevity, we forgo the classic definition, which justifies the names, and give instead a striking characterization due to Kirchberg. The reader is directed to \cite[Chapter 13]{BO} for a longer exposition. 

\begin{definition}[{{\cite[Proposition 1.1]{Kir94}}}]\label{WEP}
Let $\F$ be any nonabelian free group. 
A $\Cstar$-algebra $B$ has the {\it Weak Expectation Property} (WEP) if and only if $\Cstar (\F)\Max B \simeq \Cstar (\F)\Min B$ canonically\footnote{For $\Cstar$-algebras $A$ and $B$, we often write $A\Max B=A\Min B$ to mean that the two $\Cstar$-tensor products are canonically isomorphic, i.e., via an isomorphism extending the identity map on the algebraic tensor product.}. A $\Cstar$-algebra is called {\it QWEP} if it is the quotient of a $\Cstar$-algebra with the WEP. 
\end{definition}

In particular, $\Cstar (\F)\Max \mathcal{B}(\H)=\Cstar (\F)\Min \mathcal{B}(\H)$ for any free group $\F$ and any Hilbert space $\H$ (\cite[Corollary 1.2]{Kir94}, see also \cite{Pis96}), and so $\mathcal{B}(\H)$ has the WEP for any Hilbert space $\H$ by the above definition\footnote{That $\mathcal{B}(\H)$ has the WEP would immediately follow from the fact that $\mathcal{B}(\H)$ is injective if we gave instead Lance's definition, \cite[Definition 2.8]{Lan73}.)}. 

Just as any $\Cstar$-algebra embeds into $\mathcal{B}(\H)$ for some Hilbert space $\H$, any unital $\Cstar$-algebra is the quotient of $\Cstar (\F)$ for some free group $\F$ 
(which can be assumed countably generated when $A$ is separable). With this observation, we arrive at one more characterization for the LP and LLP, which is well-known to experts. For a sketch of the proof, see \cite[Proposition 6.5]{CShe}.

\begin{proposition}\label{LP from F}
Let $A$ be a unital $\Cstar$-algebra and $\F$ a free group such that there exist a closed two-sided ideal $J\subset \Cstar (\F)$ and $^*$-isomorphism $\phi:A\to \Cstar (\F)/J$. Then $A$ has the LLP if and only if $\phi$ locally lifts. When $A$ is separable and $\F$ is chosen to be countably generated, then $A$ has the LP if and only if $\phi$ lifts. 
\end{proposition}
In particular, if we know that any unital $^*$-homomorphism from $A$ into a quotient $\Cstar$-algebra lifts to a ucp map, then it follows that any cpc map from $A$ into a quotient $\Cstar$-algebra lifts to a cpc map. 

\subsection{Lifting $^*$-homomorphisms}
Turning our attention to $^*$-homomorphisms, we consider another important lifting property for $\Cstar$-algebras. 

\begin{definition}
A $\Cstar$-algebra is {\it projective} if for any $\Cstar$-algebra $B$ with two-sided closed ideal $J\subset B$, any $^*$-homomorphism $\phi:A\to B/J$ lifts to a $^*$-homomorphism $\psi:A\to B$ so that $\phi=\pi\psi$ where $\pi:B\to B/J$ is the quotient map. 
\end{definition}
In other words, it is a projective object in the category of $\Cstar$-algebras with $^*$-homomorphisms. 
A projective $\Cstar$-algebra cannot be unital (because it embeds into its cone). However, a $\Cstar$-algebra $A$ is projective if and only if its unitization $\tilde{A}$ is projective in the category of unital $\Cstar$-algebras with unital $^*$-homomorphisms (\cite[Proposition 2.5]{Bla98}). 
One unitally projective $\Cstar$-algebra of particular interest to us is the universal unital $\Cstar$-algebra generated by a contraction, which we denote by $\A$. (That $\A$ is unitally projective follows from the fact that any contraction in a quotient $\Cstar$-algebra lifts to a contraction.) 

As an immediate corollary to Proposition \ref{LP from F}, we have that projectivity implies the (L)LP. 
\begin{corollary}\label{Proj implies LLP}
Every projective $\Cstar$-algebra has the LLP. Every separable projective $\Cstar$-algebra has the LP.
\end{corollary}

Projectivity is an extremely strong and equally rare property, and so there has been significant interest in various weaker notions of projectivity. Our focus will be on $\Cstar$-algebras with associated conditionally projective maps. 

\begin{definition}[\cite{ELP}]\label{ELP}
Let $A$ and $C$ be $\Cstar$-algebras. We say a $^*$-homomorphism $\alpha:C\to A$ is {\it conditionally projective} if 
given a commutative diagram (with $^*$-homomorphisms)
\[
\begin{tikzcd}
C \arrow[swap]{d}{\alpha}\arrow{r}{\rho} & B\arrow{d}{\pi}\\
A \arrow[swap]{r}{\phi} & B/I,\\
\end{tikzcd}
\]
there exists a $^*$-homomorphism $\psi:A\to B$ so that the following diagram commutes
\[
\begin{tikzcd}
C \arrow[swap]{d}{\alpha}\arrow{r}{\rho} & B\arrow{d}{\pi}\\
A \arrow[swap]{r}{\phi}\arrow{ur}{\psi} & B/I.
\end{tikzcd}
\]
\end{definition}
In other words, whenever we can lift $\phi\circ\alpha$ to a $^*$-homomorphism, we can also lift $\phi$ to a $^*$-homomorphism. 
Notice that a $\Cstar$-algebra $A$ is projective exactly when the unital $^*$-homomorphism $\C\to \tilde{A}$ is conditionally projective.

In the case where $C=\C^N$ (and all maps and algebras are unital), whenever $B\to B/I$ is a quotient such that orthogonal projections (summing to one) lift to orthogonal projections (summing to one), we know that any (unital) $^*$-homomorphism $A\to B/I$ lifts to a (unital) $^*$-homomorphism $A\to B$. 


\begin{example}[\cite{LS14}]
Let $p(z)=a\sum_1^N (z-a_i)^{k_i}\in \C[z]$ be a nonconstant polynomial with complex coefficients, and let $C>0$ so that $p$ has a root of modulus strictly less than $C$. Write $A=\Cstar \langle x: \|x\|\leq C,\ p(x)=0\rangle$ for the associated universal $\Cstar$-algebra. 
Loring and Shulman show in \cite{LS14} that there exists a conditionally projective map $\C^{N}\to \tilde{A}$ (see \cite[Remark 12]{LS14}). 

In other words, if $y$ is an element in some unital quotient $\Cstar$-algebra satisfying $\|y\|\leq C$ and $p(y) = 0$, then one can assign to $y$, in a canonical and functorial way, a collection of orthogonal projections $\{p_1,...,p_{N}\}\in \Cstar (y,1)$ that sum to one such that $y$ lifts to an element $Y$ satisfying $\|Y\|\leq C$ and $p(Y)=0$ if and only if these projections lift to a collection of orthogonal projections $\{P_1,...,P_{N}\}$ that also sum to one. 

\end{example}

\subsection{Multiplier Algebras and Hilbert Modules}

We write $\K$ for the compact operators in $\mathcal{B}(\ell^2)$ and $e_{11}$ for the projection in $\mathcal{B}(\ell^2)$ onto the first coordinate of a vector in $\ell^2$. 
 For any $\Cstar$-algebra $A$, we denote its multiplier algebra by $M(A)$ and its corona algebra $M(A)/A$ by $Q(A)$. (See \cite{Bla98}, \cite{Lor97}, or \cite{Lan95} for preliminary definitions and standard results on multiplier and corona algebras for $\Cstar$-algebras.)
For any unital $\Cstar$-algebra $A$ and $x\in M(\s A)$, we denote by $x_{11}$ the image of $x$ under the compression $$(e_{11}\otimes 1_A)x(e_{11}\otimes 1_A)$$ composed with the identification $e_{11}\otimes A\simeq A$. 

For a $\Cstar$-algebra $B$ and Hilbert $B$-modules $E$ and $F$, we denote by $\L(E,F)$ the space of adjointable bounded linear maps from $E$ to $F$, and write $\L(E)=\L(E,E)$. Given a Hilbert $B$-module $E$, we denote the Hilbert $B$-module $\bigoplus_{j=1}^\infty E$ by $E^\infty$. 
The most basic example of a Hilbert $B$-module is the $\Cstar$-algebra $B$ itself with inner product $\langle a| b\rangle=a^*b$. 
Another fundamental example is $\H_B:=\H\otimes B$ where $\H$ is a separable infinite-dimensional Hilbert space. 
For our last example of a Hilbert $B$-module, let $A$ be a unital $\Cstar$-algebra, $E$ a Hilbert $B$-module, and $\phi:A\to \L(E)$ a ucp map. We define the Hilbert $B$-module $A\otimes_\phi E$ to be the separation and completion of the $B$-module $A\odot E$ with respect to the inner product given on simple tensors by 
$$\langle a_1\otimes b_1| a_2\otimes b_2\rangle=\langle b_1| \phi(a_1^*a_2)b_2\rangle.$$ 
See \cite[Chapter 5]{Lan95} for a more thorough construction of $A\otimes_\phi E$, or just follow Kasparov's construction for $E=\H_B$ in the proof of Theorem 3 in \cite{Kas80}. 

We say a Hilbert $B$-module $E$ is {\it countably generated} if there exists a countable set $\{x_n\}_n\subset E$ such that the set $\{\sum_n x_nb_n| b_n\in B\}$ is dense in $E$. We say $E$ is {\it full} if $\langle E| E\rangle$ is dense in $B$. 
(For preliminary definitions and standard results on Hilbert $B$-modules, see \cite{Lan95}.)

\section{Conditional Projectivity and the Local Lifting Property}\label{Cond Proj and LLP}
The goal of this section is to show that any separable $\Cstar$-algebra $A$ with a conditionally projective map $\C^N\to A$ has the LLP (Theorem \ref{condproj}). The main idea is to adapt Kirchberg's proof in \cite[Lemma 3.3]{Kir93} that $\Cstar (\F_\infty)$ has the LP to show that any ucp map $A\to Q(\K)$ lifts to a ucp map $A\to M(\K)$. By Ozawa's characterization of the LLP (Proposition \ref{Ozawa}), this will suffice to show that $A$ has the LLP. 
First, we must recall two powerful theorems that will be crucial for the main results in this section. 

\subsection{Dilation and Extension Theorems}\label{Section: Theorems}

The first of the two theorems is the noncommutative Tietze Extension Theorem. It was first proved for separable $\Cstar$-algebras in \cite[Theorem 4.2]{APT}. For the $\sigma$-unital case, the reader is referred to \cite[Proposition 6.8]{Lan95} or \cite[Theorem 9.2.1]{Lor97}, both of which follow the argument for separable $\Cstar$-algebras in \cite[Proposition 3.12.10]{Ped79}.
 
\begin{theorem}[Noncommutative Tietze Extension Theorem]\label{ncT}
Let $A$ and $B$ be $\Cstar$-algebras and $\pi:B\to A$ a surjective $^*$-homomorphism. Then $\pi$ extends uniquely to a $^*$-homomorphism $\hat{\pi}:M(B)\to M(A)$ where $\hat{\pi}(x)\cdot\pi(b)=\pi(x\cdot b)$ for all $b\in B$. Moreover, if $B$ is $\sigma$-unital, then $\pi$ is surjective.
\end{theorem}

When $B$ is a unital $\Cstar$-algebra with closed two-sided ideal $J$ and quotient map $\pi:B\to B/J$, we apply the noncommutative Tietze Extension Theorem to 
$$\id_\K\otimes q: \s B \to \s B/J$$ to get a surjection $M(\s B)\to M(\s B/J)$, which we also call $\hat{\pi}$. Note that for each $x\in M(\s B)$, we have $(\hat{\pi}(x))_{11}=\pi(x_{11})$. 




The second key theorem is the Kasparov-Stinespring Dilation Theorem for $\Cstar$-algebras. 
Because we will be more interested in dilating a ucp map $\phi:A\to B$, we will actually use the following corollary to (the proof of) \cite[Theorem 3]{Kas80}. While it is certainly known to experts, we cannot find an exact reference, and so we give its statement and a quick proof here. The proof is basically exactly Kasparov's in \cite{Kas80} except that we exchange the Hilbert module $A\otimes_\phi \H_B$ for $A\otimes_\phi B$ to compensate for the fact that our ucp map $\phi:A\to B$ is (highly) degenerate when viewed as a map into $M(\s B)$. 

\begin{corollary}[Kasparov]\label{Kasparov-Stinespring}
Let $A$ and $B$ be unital $\Cstar$-algebras with $A$ separable, and let $\phi:A\to B$ be a ucp map. Then there exists a unital $^*$-homomophism $\Phi:A\to M(\s B)$ such that for each $a\in A$,
$$(e_{11}\otimes 1_B)\Phi(a)(e_{11}\otimes 1_B)=e_{11}\otimes\phi(a).$$
In particular, $\Phi(a)_{11}=\phi(a)$. 
\end{corollary}

\begin{proof}
Let $E=A\otimes_\phi B$. Since $\phi$ is unital and $A$ is separable, $E$ is a countably generated and full Hilbert $B$-module. By \cite[Theorem 1.9]{MP84}, we have $E^\infty\simeq \H_B$. Let $\pi_1:A\to \L(E)$ be the unital $^*$-homomorphism induced by the left action of $A$ on $A\odot B$, i.e. for all $a'\in A$ and $a\otimes b\in A\otimes B$. 
$$\pi_1(a')(a\otimes b)=a'a\otimes b,$$
and similarly let $\pi_\infty:A\to \L(E^\infty)$ be the unital $^*$-homomophism such that,
$$\pi_\infty(a)((a_n\otimes b_n)_n)=(aa_n\otimes b_n)_n$$
for all $a\in A$ and $(a_n\otimes b_n)_n\in E^\infty$. Note that $\pi_\infty$ is a unital $^*$-homomorphism. 

Define $W\in \L(B,E)$ by $Wb=1_A\otimes b$ and $W^*(a\otimes b)=\phi(a)b$. (The argument that $W^*$ extends to $E$ is the same as in \cite[Theorem 3]{Kas80}.) 
Note that for all $a\in A$ and $b\in B$, we have $W^*\pi_1(a)W(b)=\phi(a)b$, so $W^*\pi_1(a)W=\phi(a)\in \L(B)=B$. 

Since $W^*W=1_{\L(B)}$, it follows that $WW^*\in \L(E)$ is a projection and $W\in \L(B, WW^*E)$ is a unitary. By Kasparov's stabilization theorem (\cite[Theorem 2]{Kas80}), we have an isomorphism 
$$((1-WW^*)E)\oplus E^\infty\simeq ((1-WW^*)E)\oplus \H_B\simeq \H_B,$$
which is implemented by some unitary $U\in \L(\H_B,((1-WW^*)E)\oplus E^\infty)$. 
Then $V:=W\oplus U$ is a unitary implementing an isomorphism 
$$B\oplus \H_B\simeq (WW^*E)\oplus ((1-WW^*)E)\oplus E^\infty\simeq E \oplus E^\infty.$$
So, $V^*\pi_\infty(\cdot)V=(W^*\pi_1(\cdot)W)\oplus(U^*\pi_\infty(\cdot) U):A\to \L(B\oplus \H_B)$ is a unital $^*$-homomorphism. 
Let $\Phi:A\to M(\s B)$ be the composition of $V^*(\pi_\infty)V$ with the $^*$-isomophism $\Psi:\L(B\oplus \H_B)\to M(\s B)$, which maps the projection onto the first summand to $e_{11}\otimes 1_B$. Then $\Phi$ is a unital $^*$-homomorphism, and 
\begin{align*}
    (e_{11}\otimes 1_B)\phi(a)(e_{11}\otimes 1_B)&=\Psi((W^*\pi_1(a)W)\oplus 0_{\H_B})\\
    &=\Psi(\phi(a)\oplus 0_{\H_B})=e_{11}\otimes \phi(a)
\end{align*}
for each $a\in A$.
\end{proof}

\subsection{Lifting cp maps by lifting $^*$-homomorphisms}

First, we flesh out the characterization for the LLP for separable $\Cstar$-algebras that is implicit in the combination of Kirchberg's proof of \cite[Lemma 3.3]{Kir94} and Ozawa's characterization of the LLP (\cite[Proposition 3.13]{Oza04}), which we recounted in Proposition \ref{Ozawa}. 

\begin{proposition}\label{LLP}
Let $\pi:M(\K)\to Q(\K)$ be the quotient map and $\hat{\pi}:M(\s M(\K))\to M(\s Q(\K))$ the surjective extension of $\id_K\otimes \pi$ guaranteed by the noncommutative Tietze Extension Theorem (Theorem \ref{ncT}). 
A separable unital $\Cstar$-algebra $A$ has the LLP if and only if any unital $^*$-homomorphism $\rho:A\to M(\K\otimes Q(\K))$ lifts to a ucp map $\theta:A\to M(\K\otimes M(\K))$ such that $\hat{\pi}\theta=\rho$.
\end{proposition}
  
\begin{proof}
Let $A$ be a separable unital $\Cstar$-algebra.
First, we assume $A$ has the LLP. Since $M(\K)=\mathcal{B}(\ell^2)$ 
has the WEP and $\K$ is nuclear, $\K\otimes M(\K)$ also has the WEP. Hence, \cite[Proposition 5.5]{Kir93} tells us 
that $M(\K\otimes M(\K))$ also has the WEP. By \cite[Corollary 3.12]{Oza04} (which was inspired by arguments in \cite{Kir93}), this implies that any ucp map $\rho:A\to M(\K\otimes Q(\K))$ is ucp liftable. In particular any unital $^*$-homomorphism is ucp liftable. 

On the other hand, to show that $A$ has the LLP, by \cite[Proposition 3.13]{Oza04} it suffices to lift any given ucp map $\phi:A\to Q(\K)$ to a ucp map $\psi:A\to M(\K)$. To that end, fix a ucp map $\phi:A\to Q(\K)$. Corollary \ref{Kasparov-Stinespring} then gives a unital $^*$-homomorphism $\Phi:A\to M(\s Q(\K))$ such that $(\Phi(a))_{11}=\phi(a)$ for all $a\in A$. By assumption, there is a ucp map $\theta:A\to M(\K\otimes M(\K))$ such that $\hat{\pi}\theta=\Phi$. Then the map $\psi:A\to M(\K)$ given by $\psi(a)=(\theta(a))_{11}$ is ucp. 
We compute for each $a\in A$,
\begin{align*}
    \pi(\psi(a))&=\pi((\theta(a))_{11})=(\hat{\pi}(\theta(a)))_{11}=(\Phi(a))_{11}=\phi(a).
\end{align*}
Visually, we have the following commutative diagram. 
\[
\begin{tikzcd}
 &M(\s M(\K))\arrow{d}{\hat{\pi}}\arrow{dr}{(1,1)}&&\\
&M(\s Q(\K)) \arrow{dr}{(1,1)}& M(\K)\arrow{d}{\pi}&\\
A \arrow[bend left]{uur}{\theta}\arrow{ur}{\Phi} \arrow{rr}{\phi}&& Q(\K).& \qedhere\\
\end{tikzcd}
 \]

\end{proof}

Given a $\Cstar$-algebra $A$ with closed two sided ideal $I$, we define 
$$M(A,I):=\{x\in M(A): (x\cdot A)\cup (A\cdot x)\subset I\}$$
as in \cite{Gab16}.
It is straightforward to show that $M(A,I)$ is a strictly closed two-sided ideal in $M(A)$ (in fact the strict closure of $I$ in $M(A)$). 
The following proposition is surely known to experts. We include a proof for lack of a citation. 

\begin{lemma}\label{kernel}
Let $A$ be a $\Cstar$-algebra, $I$ a closed two sided ideal in $A$. Then we have the following.
\begin{enumerate}
    \item The canonical $^*$-homomorphism $\sigma:M(A,I)\to M(I)$ is injective.
    \item The image of $\sigma$ is a hereditary $\Cstar$-subalgebra of $M(I)$. 
    \item The kernel of the extension $\hat{\pi}:M(A)\to M(A/I)$ of the map $\pi:A\to A/I$ is $M(A,I)$.
Moreover, when $A$ and $A/I$ are $\sigma$-unital, we have an exact sequence 
$$0\to M(A,I)\to M(A)\to M(A/I)\to 0.$$

\end{enumerate}
\end{lemma}

\begin{proof}
(i) Since $I$ is an ideal in $M(A,I)$, there is a unique map $\sigma: M(A,I)\to M(I)$ given by $\sigma(x)\cdot b=x\cdot b$ for all $b\in I$. To show this map is injective, it is sufficient to show that $I$ is essential in $M(A,I)$. Let $x\in M(A,I)$ and suppose $x\cdot b=b\cdot x=0$ for all $b\in I$. If $x\neq 0$, then there exists $a\in A$ such that $x\cdot a\neq 0$. Then $x\cdot a\in I$, and so there exists $b\in I$ such that $0\neq (x\cdot a)b=x\cdot (ab)$, which is a contradiction since $ab\in I$.

(ii) To show that $\sigma(M(A,I))$ is hereditary, it suffices to show that $$\sigma(M(A,I))M(I)\sigma(M(A,I))\subset \sigma(M(A,I)).$$ Let $x_1,x_2\in M(A,I)$ and $y\in M(I)$. Define the multiplier $z\in M(A)$ by 
\begin{align*}
    z\cdot a:=x_1\cdot(y\cdot(x_2\cdot a)), \hspace{.5 cm} a\cdot z:=((a\cdot x_1)\cdot y)\cdot x_2
\end{align*}
for each $a\in A$. Then $z\in M(A,I)$, and for all $b\in I$, 
\begin{align*}
\sigma(z)\cdot b&=z\cdot b=x_1\cdot(y\cdot(x_2\cdot b))\\
&=\sigma(x_1)\cdot (y\cdot (\sigma(x_2)\cdot b))\\
&=\sigma(x_1)y\sigma(x_2)\cdot b  
\end{align*}
and similarly $b\cdot \sigma(z)= b\cdot \sigma(x_1)y\sigma(x_2)$. Hence $\sigma(z)=\sigma(x_1)y\sigma(x_2)\in \sigma(M(A,I))$. 

(iii) Let $y\in M(A)$. Then 
$\hat{\pi}(y)=0$ if and only if $\hat{\pi}(y)\cdot \pi(a)=\pi(y\cdot a)=0$ and $\pi(a)\cdot \hat{\pi}(y)=\pi(a\cdot y)=0$ for all $a\in A$ if and only if $y\cdot a, a\cdot y\in I$ for all $a\in A$ if and only if $y\in M(A,I)$. 

If $A$ and $I$ are $\sigma$-unital, then the noncommutative Tietze extension theorem tells us $\hat{\pi}$ is surjective. Hence, we have the desired short exact sequence. 
\end{proof}

\begin{lemma}\label{Gabe}
Any collection $p_1,...,p_n$ of (finitely many) pairwise orthogonal,
non-zero projections in $M(\s Q(\K))$ lift to pairwise orthogonal
projections in $M(\s M(\K))$. 
Moreover, if $p_1,...,p_n$ sum to $1$, then we can arrange for their lifts to do the same. 
\end{lemma}

\begin{proof}
By \cite[Theorem 4.6]{Lor91} (which is \cite[Proposition 2.6]{AP} with norm adjustments), we can lift $p_1,...,p_n$ to pairwise orthogonal positive contractions $x_1,...,x_n\in M(\s M(\K))$. For $i=1,...,n$ define 
\begin{align*}
    I_i&:=\overline{x_iM(\s M(\K) , \s \K)x_i},\\
    A_i&:=\overline{x_iM(\s M(\K))x_i},\ \text{and}\\
    B_i&:= p_iM(\s Q(\K))p_i.
\end{align*}
Then for each $i=1,...,n$ we have a short exact sequence 
$$0\to I_i\to A_i\to B_i\to 0.$$
Since the $x_i$ are pairwise orthogonal, if we can lift each $p_i$ to projections in $A_i$, these lifts will be pairwise orthogonal. 

Fix $1\leq i\leq n$. Since $I_i$ is a hereditary $\Cstar$-subalgebra of $M(\s M(\K), \K\otimes\K)$, 
it follows from Lemma \ref{kernel} that $I_i$ embeds as a hereditary $\Cstar$-subalgebra of $M(\s \K)$. Since $M(\s \K)\simeq \mathcal{B}(\ell^2)$ is real rank zero, and since real rank zero passes to hereditary subalgebras (\cite[Corollary 2.8]{BP91}), it follows that $I_i$ has real rank zero. If we can show that $K_0(B_i)=0$, then we can use Corollary 3.15 from \cite{BP91} to conclude that every projection in $B_i$ (in particular $p_i$) lifts to a projection in $A_i$. 

Because $\s Q(\K)$ is simple, $\sigma$-unital, and purely infinite, it follows that $\frac{M(\s Q(\K))}{\s Q(\K)}$ is simple by \cite[Theorem 3.2]{Ror91} or alternatively by \cite[Theorem 3.8]{Lin91}; moreover, $\s Q(\K)$ is the unique closed two sided ideal in $M(\s Q(\K))$.  
So $B_i$ is a full hereditary subalgebra of either $M(\s Q(\K))$ or $\s Q(\K)$, and all three are $\sigma$-unital. Then by Brown's stable isomorphism theorem (\cite[Theorem 2.8]{Bro77}), either $\s B_i\simeq \s M(\s Q(\K))$ or $\s B_i\simeq \s Q(\K)$, which means we have either $K_0(B_i)\simeq K_0(Q(\K))=0$ or $K_0(B_i)\simeq K_0(M(\s Q(\K))$. Since multiplier algebras of stable $\Cstar$-algebras have trivial K-theory (\cite[Proposition 12.2.1]{Bla98}), either way, we have $K_0(B_i)=0$ and can now invoke \cite[Corollary 3.15]{BP91} to finish the argument. 

Finally, if $\sum_{i=1}^n p_i=1$, and $q_1,...,q_n$ are orthogonal lifts of $p_1,...,p_n$, then $q_1, q_2,...,q_{n-1}, (1-\sum_{i=1}^n q_i)$ are orthogonal lifts of $p_1,...,p_n$ that sum to $1$.
\end{proof}

\begin{corollary}\label{lifting CN}
For any $N\geq 1$, any unital $^*$-homomorphism $\phi:\C^N\to M(\s Q(\K))$ lifts to a unital $^*$-homomorphism $\psi:\C^N\to M(\s M(\K))$. 
\end{corollary}

\begin{theorem}\label{condproj}
Suppose $A$ is a separable unital $\Cstar$-algebra with a conditionally projective unital $^*$-homomorphism $\alpha:\C^N\to A$ for some $N\geq 1$. Then $A$ has the LLP.
\end{theorem}

\begin{proof}
Let $\rho:A\to M(\s Q(\K))$ be a unital $^*$-homomorphism. Then $\rho\alpha:\C^N\to M(\s Q(\K))$ is a unital $^*$-homorphism, and hence has a lift by Corollary \ref{lifting CN}. By assumption, this implies $\rho$ lifts to a unital $^*$-homomorphism, which completes the proof by Proposition \ref{LLP}.
\end{proof}

\begin{corollary}\label{polyLP}
Let $C>0$ and $p\in \C [z]$ any non-constant polynomial with a root of modulus strictly less than $C$. Then $A=\Cstar \langle x: \|x\|\leq C,\ p(x)=0\rangle$ has the LLP. \end{corollary}
The special case for $p(x)=x^n$ follows from the fact that the associated $\Cstar$-algebra is projective (\cite{Shu08}) and Proposition \ref{Proj implies LLP}.
\begin{proof}
From \cite[Remark 12]{LS14}, we know there is a conditionally projective map $\C^{N}\to \tilde{A}$ for some $N\in \N$. 
\end{proof}

As we have already remarked, any projective or nuclear $\Cstar$-algebra has the LLP (using Corollary \ref{Proj implies LLP} and the Choi-Effros Lifting theorem, respectively). However, ``most" of the $\Cstar$-algebras from Corollary \ref{polyLP} do not fall under either of these labels. From \cite{Shu08}, we know that $A=\Cstar \langle x: \|x\|\leq 1,\ p(x)=0\rangle$ is projective if and only if $p(x)=x^n$ for some $n\geq 1$. 
As for nuclearity, although some of these algebras are nuclear, e.g. $\Cstar \langle x: \|x\|\leq 1,\ x^2=0\rangle\simeq C_0((0,1], \mathbb{M}_2(\C))$ (\cite{Lor97}), most are not. 

\begin{proposition}
For each $\lambda\in \C\backslash\{0\}$ and $C>|\lambda|$, there exists an $N>0$ such that $\Cstar \langle x: \|x\|\leq C,\ p(x)=0\rangle$ is not nuclear (and in fact non-exact) whenever $(x-\lambda)^N\ |\ p(x)$.
\end{proposition}

\begin{proof}
Fix $\lambda\neq 0$ and $C>|\lambda|$. After scaling, we can assume $C=1$. For each $n\geq 1$, write $A_{\lambda, n}=\Cstar _u\langle x: \|x\|\leq C,\ (x-\lambda)^n=0\rangle$. If $(x-\lambda)^n\ |\ p(x)$, then $\Cstar _u\langle x: \|x\|\leq C,\ p(x)=0\rangle$ surjects onto $A_{\lambda, n}$, so it suffices to prove that $A_{\lambda,N}$ is not exact for some $N$. (This is because quotients of exact $\Cstar$-algebras are always exact, a deep result due to Kirchberg-- see \cite[Corollary 9.4.4]{BO} for an argument.) 

Recall that $\A$ denotes the universal unital $\Cstar$-algebra of a contraction. 
Though it is not explicitly stated, in the proof of Theorem 5.11 in \cite{CShe}, there is an intermediate result generalizing \cite[Theorem 5.9]{CShe} which says that any faithful unital representation $\rho:\A\to \mathcal{B}(\H)$ asymptotically factors through $\{A_{\lambda, n}\}_n$, i.e. there exist a sequence of $^*$-homomorphisms $\psi_n:\A\to A_{\lambda, n}$ and $\phi_n:A_{\lambda, n}\to \mathcal{B}(\H)$ so that $\phi_n\psi_n\to \rho$ pointwise in norm. If $A_{\lambda, n}$ were exact for more than finitely many $n$, then it would follow that $\phi_n$ are nuclear for all $n$, and hence $\rho$ would also be nuclear. This would imply that $\A$ is exact, but this is a contradiction (for instance, we know $\Cstar (\F_2)$ embeds into $\A$ -- see e.g. \cite[Theorem 6.10]{CShe}). 
\end{proof}

We do not know if the $\Cstar$-algebras from Corollary \ref{polyLP} have the LP. Because each is unitized and singly generated, it would suffice to prove  that the identification with a quotient of $\Cstar (\F_2)$ lifts to a cpc map. While we cannot say in general whether this map has a cpc lift, we can say that it has no $^*$-homomorphism lift. Indeed, the $\Cstar$-algebras from Corollary \ref{polyLP} have nontrivial projections, such as those named in \cite[Corollary 6]{LS14}, but $\Cstar (\F_2)$ has no nontrivial projections (\cite[Theorem 1]{Cho80}).


\section{Pythagorean $\Cstar$-algebras}\label{Pn}
In this section, we prove that for each $n>1$, the Pythagorean $\Cstar$-algebra 
$$\mathcal{P}_n=\Cstar \langle x_1,...,x_n : \textstyle{\sum} x_i^*x_i=1\rangle$$
has the LP. (This also holds for $n=1$, but in this case $P_n$ is the Toeplitz algebra, which is nuclear and has the LP by the Choi-Effros lifting theorem.) The key idea will again be a dilation argument, but this time we will be dilating elements of a $\Cstar$-algebra instead of maps between $\Cstar$-algebras. 

\begin{proposition}\label{Pn has LP}
For each $n> 1$, $\mathcal{P}_n$ has the LP. 
\end{proposition}

\begin{proof}
Let $\pi:\Cstar (\F_\infty)\to P_n$ be a surjection. By Proposition \ref{LP from F}, it suffices to lift $\id_{\P_n}$ to a ucp map into $\Cstar (\F_\infty)$. Since $\sum x_i^*x_i=1$, in particular, we have that $\|\sum x_i^*x_i\|\leq 1$, which is a liftable relation. (This is usually attributed to folklore/functional calculus; for a proof of a more general fact, see \cite[Theorem 2.3]{LS12}.) Let $y_1,...,y_n\in \Cstar (\F_\infty)$ be lifts of $x_1,...,x_n$ satisfying $\|\sum y_i^*y_i\|\leq 1$, and 
let $d=(1-\sum y_i^*y_i)^{1/2}\in \Cstar (\F_\infty)$. Now, define $\y_1,...,\y_n\in \mathbb{M}_2(\Cstar (\F_\infty))$ by 
$$\y_1=\begin{pmatrix} y_1 & 0\\ d & 0\end{pmatrix}, \hspace{.25 cm} \y_2=\begin{pmatrix} y_2 & 0\\ 0 & 1\end{pmatrix}, \hspace{.25 cm} \text{and} \hspace{.25 cm} \y_i=\begin{pmatrix} y_i & 0 \\ 0 & 0\end{pmatrix}$$ for each $3\leq i\leq n$. 
Then, we have 
\begin{align*}
    \textstyle{\sum}_{i=1}^n \y_i^*\y_i&= \begin{pmatrix} y_1^*y_1+ d^2 & 0 \\ 0 & 0\end{pmatrix}+\begin{pmatrix} y_2^*y_2 & 0\\ 0 & 1\end{pmatrix}+\begin{pmatrix}\textstyle{\sum}_{i=3}^n y_i^*y_i & 0 \\ 0 & 0 \end{pmatrix}=\begin{pmatrix} 1 & 0 \\ 0 & 1\end{pmatrix},
\end{align*}
and so the map $x_i\mapsto \y_i$ induces a $^*$-homomorphism $\phi:P_n\to \Cstar (\y_1,...,\y_n)\in \mathbb{M}_2(\Cstar (\F_\infty))$. Let $\rho:\mathbb{M}_2(\Cstar (\F_\infty))\to \Cstar (\F_\infty)$ be the compression onto the (1,1) coordinate. It remains to show that $\id_{P_n}=\pi\rho\phi$. \footnote{It may be tempting to conclude this from the fact that $\pi\rho\phi(x_i)=x_i$ for each $i$, but remember that $\pi\rho\phi$ need not be a $^*$-homomorphism.} 

We write ${\bf x, y, \y}$ for the respective $n$-tuples.
Since all the maps are bounded and linear, it suffices to check that for a $^*$-monomial $q$ in $n$ variables 
\begin{equation}\label{q}
    \pi\rho\phi(q({\bf x}))=(\pi\rho)(q({\bf \y}))=\pi(q({\bf y}))=q(\pi(y_1),...,\pi(y_n))=q({\bf x}).
\end{equation}
The second equality is the only one that we need to verify, and it will follow from showing that $\rho(q({\bf \y}))\in q({\bf y})+\ker(\pi)$ for any $^*$-monomial $q$ in $n$ variables. We do so inductively. 

Clearly this holds for each $\y_1,...,\y_n,\y_1^*,...\y_n^*,$ and $1$. So, suppose $q$ is a $^*$-monomial for which $(\pi\rho)(q({\bf \y}))=\pi(q({\bf y}))$. We check that this still holds for $\y_i q({\bf \y}), \y_i^*q({\bf \y}), q({\bf \y})\y_i,$ and $q({\bf \y})\y_i^*$ for each $1\leq i\leq n$. Let $r\in \ker(\pi)$ so that $\rho(q({\bf \y}))=q({\bf y})+r$, and write $$q({\bf \y})=\begin{pmatrix} q({\bf y})+r & s_{12}\\ s_{11} & s_{22}\end{pmatrix}$$ for some $s_{12}, s_{21}, s_{22}\in \Cstar (\F_\infty)$. Then we compute 
\begin{align*}
        \rho(\y_i q({\bf \y}))&=y_iq({\bf y})+y_ir,& \rho(q({\bf \y})\y_i)&=q(\y)y_i+ry_i\\ 
    \rho(\y_i^*q({\bf \y}))&=y_i^*q({\bf y})+y_i^*r,& \rho(q({\bf \y})\y_i^*)&=q(\y)y_i^*+ ry_i^*,
\end{align*}
for $2\leq i\leq n$, and 
\begin{align*}
    \rho(\y_1 q({\bf \y}))&=y_1q({\bf y})+y_1r,& \rho(q({\bf \y})\y_1)&=q(\y)y_1+ry_1+s_{12}d\\  
    \rho(\y_1^*q({\bf \y}))&=y_1^*q({\bf y})+y_1^*r+ds_{21},& \rho(q({\bf \y})\y_1^*)&=q(\y)y_1^*+ ry_1^*.\\
\end{align*} 
Since $r,d\in \text{ker}(\pi)$, equation (\ref{q}) follows, and we have $\id_{P_n}=\pi\rho\phi$ as desired.
\end{proof}

In particular, any $^*$-homomorphism from $\P_n$ into the Calkin algebra lifts to a cpc map, and so the semigroup of extensions of $\P_n$ by $\K$ is also a group. 

\begin{corollary}\label{Ext}
For each $n>1$, the semigroup $Ext(\P_n)$ is a group. 
\end{corollary}

\section{$\Cstar$-algebras which characterize Lance's WEP}\label{CEP}
In \cite{Kir93}, Kirchberg established the following tensorial duality between the LLP and Lance's WEP. 

\begin{proposition}[{{\cite[Proposition 1.1]{Kir94}}}]\label{Kir1.1}
Let $\F$ be any nonabelian free group and $\H$ an infinite dimensional Hilbert space. For $\Cstar$-algebras $A$ and $B$, we have the following.
\begin{enumerate}
    \item If $A$ has the LLP and $B$ has the WEP, then $A\Max B= A\Min B$.
    \item $B$ has the WEP if and only if $\Cstar (\F)\Max B= \Cstar (\F)\Min B$.
    \item $A$ has the LLP if and only if $A\Max \mathcal{B}(\H)=A\Min \mathcal{B}(\H)$.
\end{enumerate}
\end{proposition}
The reader will recognize the second item as our definition for the WEP in Section \ref{preliminaries}. In fact, a $\Cstar$-algebra $A$ is said to {\it characterize the WEP} (\cite{FKPT}) if for any $\Cstar$-algebra $B$, $B$ has the WEP if and only if $A\Max B=A\Min B$. By Proposition \ref{Kir1.1}, $\Cstar (\F)$ characterizes the WEP for any nonabelian free group $\F$. It turns out that several other universal $\Cstar$-algebras also characterize the WEP by virtue of having the LLP and containing a relative weakly injective copy of $\Cstar (\F_2)$. 

\begin{definition}\label{rwi}
For a $\Cstar$-algebra $B$ with $\Cstar$-subalgebra $A$, we say $A$ embeds {\it relatively weakly injectively} into $B$ if for any $\Cstar$-algebra $C$, 
$$A\Max C\subset B\Max C.$$
\end{definition}

If a $\Cstar$-algebra $A$ has the LLP, then for any $\Cstar$-algebra $B$ with the WEP, it follows from Proposition \ref{Kir1.1} that $A\Max B= A\Min B$. On the other hand, suppose $\Cstar (\F_2)$ embeds relatively weakly injectively into $A$, and $A\Max B=A\Min B$ for some $\Cstar$-algebra $B$. Then the canonical surjection $\Cstar (\F_2)\Max B\to \Cstar (\F_2)\Min B$ is injective, and Proposition \ref{Kir1.1} and the following diagram imply that $B$ has the WEP. 
\[
\begin{tikzcd}
\Cstar (\F_2)\Max B \arrow[two heads]{d} \arrow[r,phantom, "\rotatebox{0}{$\subseteq$}", description] & A\Max B\arrow[d,phantom, "\rotatebox{90}{$=$}", description]\\
\Cstar (\F_2)\Min B \arrow[r,phantom, "\rotatebox{0}{$\subseteq$}", description] & A\Min B\\
\end{tikzcd}
\]
 Hence, a $\Cstar$-algebra characterizes the WEP if it has the LLP and contains a relatively weakly injective copy of $\Cstar (\F_2)$. 
As for the converse, it follows readily from Proposition \ref{Kir1.1} that any $\Cstar$-algebra that characterizes the WEP has the LLP (by virtue of having a unique $\Cstar$-tensor product with $\mathcal{B}(\H)$). However, it is unknown whether any $\Cstar$-algebra that characterizes the WEP must contain a copy of $\Cstar (\F_2)$. Some interesting partial results were given in \cite{FKPT}. 

Because the universal unital contraction algebra $\A$ is projective and singly generated, it is often easier to find a relatively weakly injective copy of $\A$ in a given $\Cstar$-algebra than it is to find such a copy of $\Cstar (\F_2)$. 
In \cite[Theorem 6.10]{CShe}, the authors show that $\Cstar (\F_2)$ embeds relatively weakly injectively into $\A$. 
Since embedding relatively weakly injectively is really just a statement about norms on linear combinations of simple tensors, the relation is clearly transitive. Therefore, if $\A$ embeds relatively weakly injectively into a given $\Cstar$-algebra, then so does $\Cstar (\F_2)$. Hence, any $\Cstar$-algebra that has the LLP and contains a relatively weakly injective copy of $\A$ also characterizes the WEP.
With this observation, we can easily list new examples of $\Cstar$-algebras that characterize the WEP. 

\begin{example}\label{example} The following $\Cstar$-algebras characterize the WEP.

\noindent (1)\ The Pythagorean $\Cstar$-algebras $\P_n$. 
    
    For each $n>1$, $\P_n$ has the LLP by Proposition \ref{Pn has LP}. Let $x$ denote the generator of $\A$ and $x_1,...,x_n$ the generators of $\P_n$. Since $\A$ is a quotient of $\P_n$ via the map $x_1\mapsto x$, $x_2\mapsto \sqrt{1-x^*x}$ and $x_n\mapsto 0$, the projectivity of $\A$ gives an embedding of $\A$ into $\P_n$.
In other words, there is a conditional expectation (actually even a $^*$-homomorphism) from $\P_n$ onto this copy of $\A$ in $\P_n$. 
By \cite{Lan73}, this implies $\A$ embeds relatively weakly injectively into $\P_n$.
\vspace{.5 cm}

\noindent(2)\ The universal (unital) $\Cstar$-algebra generated by an $n$-row contraction,
\begin{equation*}
\mathcal{R}_n:=\Cstar _u\langle y_1,...,y_n: \|\textstyle{\sum} y_iy_i^*\|\leq 1\rangle,\end{equation*} 
as studied in \cite{CShe}. 

For $n\geq 1$, we know that $\mathcal{R}_n$ is projective (this is also usually attributed to folklore/functional calculus, see \cite[Theorem 2.3]{LS12}), and hence it has the LP by Lemma \ref{LP from F}. As with $\P_n$, since it surjects onto the projective $\Cstar$-algebra $\A$, it follows that $\A$ embeds relatively injectively into $\mathcal{R}_n$.
\vspace{.5 cm}

\noindent(3)\ The universal $\Cstar$-algebra generated by a partial isometry, $$\P:=\Cstar \langle x: xx^*x=x\rangle,$$ as studied in \cite{BN12}.

In \cite{BN12}, the authors show that $\A$ is isomorphic to a (full) hereditary subalgebra of $\P$. It follows from Brown's stable isomorphism theorem (\cite[Theorem 2.8]{Bro77}) that $\P$ and $\A$ are stably isomorphic. Since the LLP is preserved under tensor products with nuclear $\Cstar$-algebras (\cite{Kir93}), it follows that $\P$ has the LLP. Since hereditary subalgebras embed relatively weakly injectively (see \cite[Corollary 3.6.4]{BO}), we also have that $\A$ embeds relatively weakly injectively into $\P$. 

Note that since $\P$ has the LLP, it follows that Ext$(\P)$ is a group, just as it did for $\P_n$ in Corollary \ref{Ext}. 
\end{example}

We end by remarking on a clear connection with Connes' Embedding Problem. If any one $\Cstar$-algebra that characterizes the WEP {\it has} the WEP, then {\it all} $\Cstar$-algebras that characterize the WEP have the WEP. Indeed, suppose $A$ and $B$ both characterize the WEP, and $A$ has the WEP. Then, since $B$ has the LLP, $A\Max B=A\Min B$ by Proposition \ref{Kir1.1}, which then would imply that $B$ has the WEP.  
By a striking and well-known result of Kirchberg (\cite[Proposition 8.1]{Kir93}), $\Cstar (\F_\infty)$ has the WEP if and only if Connes' Embedding Problem has a positive solution. 

A negative solution to Connes' Embedding Problem has very recently been announced in \cite{JNVWY}, and, if correct, would imply that none of the $\Cstar$-algebras from Example \ref{example} have the WEP.


\providecommand{\bysame}{\leavevmode\hbox to3em{\hrulefill}\thinspace}
\providecommand{\MR}{\relax\ifhmode\unskip\space\fi MR }
\providecommand{\MRhref}[2]{%
  \href{http://www.ams.org/mathscinet-getitem?mr=#1}{#2}
}
\providecommand{\href}[2]{#2}

\end{document}